\documentclass[12pt]{amsart}

\usepackage{amssymb,amsmath,amsthm}

\def\RR{{\mathbb R}}
\def\ZZ{{\mathbb Z}}

\def\QQ{{\mathbb Q}}
\newtheorem{Theorem}{Theorem}[section]
\newtheorem{Lemma}[Theorem]{Lemma}
\newtheorem{Corollary}[Theorem]{Corollary}
\newtheorem{Proposition}[Theorem]{Proposition}
\theoremstyle{definition}
\newtheorem{Remark}[Theorem]{Remark}
\newtheorem{Example}[Theorem]{Example}

\def\PP{\mathcal P}
\def\QQ{\mathcal Q}
\def\eb{{\bf e}}

\begin{document}

\title{Toric rings and ideals of stable set polytopes}
\author{Kazunori Matsuda, Hidefumi Ohsugi and Kazuki Shibata}

\address{Kazunori Matsuda,
Department of Pure and Applied Mathematics, 
Graduate School of Information Science and Technology,
Osaka University, Suita, Osaka  565-0871, Japan}
\email{kaz-matsuda@ist.osaka-u.ac.jp}

\address{Hidefumi Ohsugi,
Department of Mathematical Sciences, 
School of Science and Technology, 
Kwansei Gakuin University, Sanda, Hyogo, 669-1337, Japan}
\email{ohsugi@kwansei.ac.jp}

\address{Kazuki Shibata,
Department of Mathematics,
College of Science,
Rikkyo University,
Toshima-ku, Tokyo 171-8501, Japan} 
\email{k-shibata@rikkyo.ac.jp }

\subjclass[2010]{}

\begin{abstract}
In this paper, 
we discuss the normality of the toric rings of stable set polytopes,
and the set of generators and Gr\"obner bases of toric ideals of stable set polytopes
by using the results on that of edge polytopes of finite nonsimple graphs.
In particular, for a graph of stability number two,
we give a graph theoretical characterization of the set of generators of the toric ideal of the stable set polytope, 
and a criterion to check whether the toric ring of the stable set polytope is normal or not.
One of the application of the results is an infinite family of stable set polytopes whose toric ideal is
generated by quadratic binomials and has no quadratic Gr\"obner bases.
\end{abstract}

\maketitle

\section*{Introduction}

Let $\PP \subset \RR^n$ be an integral convex polytope, i.e., a convex polytope 
each of whose vertices has integer coordinates, and
let $\PP \cap \ZZ^n = \{{\bf a}_1, \ldots, {\bf a}_m\}$.
Let $K[X,X^{-1}, t]=K[x_1,x_1^{-1},\ldots,x_n,x_n^{-1}, t]$
be the Laurent polynomial ring in $n+1$ variables over a field $K$.
Given an integer vector ${\bf a} =(a_1,\ldots,a_n) \in \ZZ^n$,
we set $X^{\bf a} t = x_1^{a_1} \cdots x_n^{a_n}t  \in K[X,X^{-1}, t]$.
Then, the {\em toric ring} of $\PP$ is the subalgebra $K[\PP]$ of $K[X,X^{-1},t]$
generated by $\{ X^{{\bf a}_1} t, \ldots,  X^{{\bf a}_m} t \}$ over $K$.
%those Laurent monomials $X^{\bf a} t = x_1^{a_1} \cdots x_n^{a_n}t $
%such that ${\bf a} =(a_1,\ldots,a_n) \in \PP \cap \ZZ^n$.
We regard $K[\PP]$ as
a homogeneous algebra
by setting each $\deg X^{{\bf a}_i} t = 1$.
The {\em toric ideal} $I_\PP$ of $\PP$ is the kernel of a surjective homomorphism 
$\pi: K[y_1,\ldots, y_m] \rightarrow K[\PP]$
defined by $\pi(y_i) =  X^{{\bf a}_i} t$ for $1 \le i \le m$.
It is known that $I_\PP$ is generated by homogeneous binomials and 
reduced Gr\"obner bases of $I_\PP$ consist of homogeneous binomials.
See, e.g., \cite{Stu}.
The following properties on 
an integral convex polytope $\PP$ have been investigated by
many papers on commutative algebra and combinatorics:
\begin{enumerate}
\item[(i)]
$\PP$ is unimodular, i.e., any triangulation of $\PP$ is unimodular\\
(The initial ideal of $I_\PP$ is generated by squarefree monomials
with respect to any monomial order);
\item[(ii)]
$\PP$ is compressed, i.e., 
any ``pulling" triangulation is unimodular\\
(The initial ideal of $I_\PP$ is generated by squarefree monomials
with respect to any reverse lexicographic order);
\item[(iii)]
$\PP$ has a regular unimodular triangulation\\
(There exists a monomial order such that 
the initial ideal of $I_\PP$ is generated by squarefree monomials);
\item[(iv)]
$\PP$ has a unimodular triangulation;
\item[(v)]
$\PP$ has a unimodular covering;
\item[(vi)]
$\PP$ is normal, i.e., $K[\PP]$ is normal.
\end{enumerate}
The hierarchy 
(i) $\Rightarrow $
(ii) $\Rightarrow $
(iii) $\Rightarrow $
(iv) $\Rightarrow $
(v) $\Rightarrow $
(vi)
is known.
See, e.g., \cite{Stu} for details.
However, the converse of each of the five implications is false.
On the other hand, the following properties on $I_\PP$ are studied by
many authors:
\begin{enumerate}
\item[(a)]
$I_\PP$ has a quadratic Gr\"obner bases;
\item[(b)]
$K[\PP]$ is Koszul algebra;
\item[(c)]
$I_\PP$ is generated by quadratic binomials.
\end{enumerate}
The hierarchy 
(a) $\Rightarrow $
(b) $\Rightarrow $
(c) 
is known.
However, the converse of each of the two implications is false.
See, e.g., \cite[Examples 2.1 and 2.2]{OHquad} and \cite{counterexamples}.

The purpose of this paper is to study 
such properties of toric rings and ideals of 
stable set polytopes of simple graphs.
Let $G$ be a finite simple graph on the vertex set $[n]=\{1,2,\ldots,n\}$
and let $E(G)$ denote the set of edges of $G$.
Given a subset $W \subset [n]$, we associate the $(0,1)$-vector
$\rho(W) = \sum_{j \in W} {\eb}_j \in {\RR}^n.$
Here, ${\bf e}_i$ is the $i$th unit coordinate vector of ${\RR}^n$. 
In particular, $\rho(\emptyset)$ is the origin of $\RR^n$.
A subset $W$ is called {\em stable} if $\{i,j\} \notin E(G)$ for all $i,j \in W$
with $i \neq j$.
Note that the empty set and each single-element subset of $[n]$ are stable.
Let $S(G)$ denote the set of all stable sets of $G$.
The {\em stable set polytope} ({\em independent set polytope}) 
of a simple graph $G$ is the $(0,1)$-polytope
$\QQ_G \subset \RR^n$ which is the convex full of $\{\rho(W) : W \in S(G)\}$.
Stable set polytopes are very important in many areas, e.g., 
optimization theory.
On the other hand, several results are known for
the toric ring $K[\QQ_G]$ and ideals $I_{\QQ_G}$ of the stable set polytope 
$\QQ_G$ of a simple graph $G$.
\begin{enumerate}
\item
For a graph $G$, the stable set polytope
$\QQ_G$ is compressed if and only if $G$ is perfect
(\cite{GPT, OHcompressed, Sul}).
\item
For a perfect graph $G$, the toric ring
$K[\QQ_G]$ is Gorenstein if and only if 
all maximal cliques of $G$ have the same cardinality
(\cite{OHspecial}).

\item
For a graph $G$, $K[\QQ_G]$ is strongly Koszul if and only if 
$G$ is trivially perfect (\cite[Theorem 5.1]{KM}).

\item
Let $G(P)$ be a comparability graph of a poset $P$.
Then, $\QQ_{G(P)}$ is called a {\em chain polytope} of $P$.
It is known that the toric ideals of a chain polytope has a squarefree quadratic initial ideal
(see \cite[Corollary 3.1]{HL}).
For example, if a graph $G$ is bipartite, then there exists a 
poset $P$ such that $G=G(P)$.
\item
Suppose that a graph $G$ on the vertex set $[n]$ is
an almost bipartite graph, i.e., there exists a vertex $v$
such that the induced subgraph of $G$ on the vertex set $[n] \setminus \{v\}$ is bipartite.
Then, $I_{\QQ_G}$ has a squarefree quadratic initial ideal
(\cite[Theorem 8.1]{EnNo}).
For example, any cycle is almost bipartite.

\item
Let $G$ be the complement of even cycle of length $2k$.
Then, the maximum degree of a minimal set of 
binomial generators of $I_{\QQ_G}$ is equal to $k$
(\cite[Theorem 7.4]{EnNo}).
\end{enumerate}
In the present paper, we study  the normality of the toric rings of stable set polytopes,
and the set of generators and Gr\"obner bases of toric ideals of stable set polytopes
by using the results on that of edge polytopes of finite (nonsimple) graphs.
Here, the {\em edge polytope} $\PP_G \subset \RR^n$ 
of a graph $G$ allowing loops and having no multiple edges
is the convex full of 
%$\{\rho(e) : e \in E(G)\}$.
%
$$
\{\eb_i + \eb_j : \{i,j\} \mbox{ is an edge of } G \}\ 
\cup \ 
\{ 2 \eb_i : \mbox{ there is a loop at } i \mbox{ in } G \}
.$$

This paper is organized as follows. 
In Section 1, fundamental properties
 of $K[\QQ_G]$ and $I_{\QQ_G}$ are studied.
In particular, the relationship between the stable set polytopes
and the edge polytopes are given (Lemma \ref{keylemma}).
In addition, it is shown that $\QQ_G$ is unimodular
if and only if the complement of $G$ is bipartite (Proposition \ref{unimodular}).
We also point out that,
by the results in \cite{EnNo}, it is easy to see that
$I_{\QQ_G}$ has a squarefree quadratic initial ideal if
 $G$ is either a chordal graph or a ring graph (Proposition \ref{chordalring}).
In Section 2, we discuss the normality of the stable set polytopes.
We prove that, for a simple graph $G$ of stability number two,
$\QQ_G$ is normal if and only if the complement of $G$ satisfies the
``odd cycle condition" (Theorem \ref{a2normal}).
Using this criterion, we construct an infinite family of normal stable set polytopes
without regular unimodular triangulations (Theorem \ref{aninfinite}).
For general simple graphs, some necessary conditions for $\QQ_G$ to be normal  
are also given.
In Section 3, we study the set of generators and Gr\"obner bases of toric ideals of stable set polytopes.
It is shown that, for a simple graph $G$ of stability number two,
the set of binomial generators of $I_{\QQ_G}$ is described by 
the even closed walk of a graph (Theorem \ref{generators}).
If $\overline{G}$ is bipartite and if $I_{\QQ_G}$ is generated by quadratic binomials,
then $I_{\QQ_G}$ has a quadratic Gr\"obner basis (Corollary \ref{KB}). 
Finally, using the results on normality, generators, and Gr\"obner bases,
we give an infinite family of nonnormal stable set polytopes
whose toric ideal is generated by quadratic binomials and has no quadratic Gr\"obner bases
(Theorem \ref{aninfinite}).

\section{Fundamental properties of the stable set polytopes}

In this section, we give some fundamental properties of $K[\QQ_G]$
and $I_{\QQ_G}$.
In particular, a relation between 
the stable set polytopes and the edge polytopes is discussed.
The {\em stability number} $\alpha(G)$ of a graph $G$
is the cardinality of the largest stable set.

\begin{Example}
\label{completegraph}
If a simple graph $G$ satisfies $\alpha(G) =1$, i.e., $G$ is a complete graph,
then $I_{\QQ_G} =\{ 0 \} $, $K[\QQ_G] = K[x_1 t, \ldots, x_nt ,t] \simeq K[y_1,\ldots,y_{n+1}]$,   and
$\QQ_G$ is a simplex.
\end{Example}

\begin{Example}
Suppose that a simple graph $G$ is not connected.
Let $G_1, \ldots , G_s$ be connected components of $G$.
Then, it is easy to see that $K[\QQ_G]$ is isomorphic to the Segre product of 
$K[\QQ_{G_1}], \ldots, K[\QQ_{G_s}]$.
\end{Example}

Thus, it is enough to study stable set polytopes of connected simple graphs $G$ such that $\alpha(G) \geq 2$.
The notion of toric fiber products \cite{tfp} is a generalization of the Segre product.
It is known \cite{EnNo} that we can apply the theory of toric fiber products
to the toric rings of stable set polytopes.
For $i = 1,2$, let $G_i$ be a simple graph on the vertex set $V_i$
and the edge set $E_i$.
If $V_1 \cap V_2$ is a clique of both $G_1$ and $G_2$, then 
we construct a new graph $G_1 \sharp G_2$
on the vertex set $V_1 \cup V_2$ and the edge set $E_1 \cup E_2$
which is called the {\em clique sum} of $G_1$ and $G_2$
along $V_1 \cap V_2$.

\begin{Proposition}
\label{tfiber}
Let $G_1 \sharp G_2$ be the clique sum of simple graphs $G_1$ and $G_2$.
Then, $I_{\QQ_{G_1 \sharp G_2}}$ is a toric fiber product of 
$I_{\QQ_{G_1}}$ and $I_{\QQ_{G_2}}$.
We can construct a set of binomial generators (or a Gr\"obner basis) of 
$I_{\QQ_{G_1 \sharp G_2}}$ from that of 
$I_{\QQ_{G_i}}$'s and some quadratic binomials.
Moreover, $K[\QQ_{G_1 \sharp G_2}]$ is normal if and only if 
both $K[\QQ_{G_1}]$ and $K[\QQ_{G_2}]$ are normal.
\end{Proposition}

\begin{proof}
This is a special case of \cite[Proposition 5.1]{EnNo}.
Note that $I_{\QQ_{G_1 \cap G_2}} = \{0\}$.
\end{proof}

A simple graph $G$ is called {\em chordal} if any induced cycle of $G$
is of length 3.
A {\em ring graph} is a graph whose block which is not a bridge or a vertex can be constructed from
a cycle by successively adding cycles of length $\ge 3$ using the edge sum construction.
Ring graphs are introduced in \cite{ring1, ring2} and the toric ideals of cut polytopes of ring graphs
are studied in \cite{NP}.

\begin{Proposition}
\label{chordalring}
Suppose that a simple graph $G$ is either a chordal graph or a ring graph.
Then, $I_{\QQ_G}$ has a squarefree quadratic initial ideal.
\end{Proposition}

\begin{proof}
It is known that a graph $G$ is chordal if and only if 
$G$ is a clique sum of complete graphs.
By the statement in Example \ref{completegraph}
and Proposition \ref{tfiber}, $I_{\QQ_G}$ has a squarefree quadratic initial ideal
if $G$ is chordal.

Suppose that $G$ is a ring graph.
Then, $G$ is a clique sum of trees and cycles.
Since trees and cycles are almost bipartite, by \cite[Theorem 8.1]{EnNo},
the toric ideal $I_{\QQ_H}$ 
has a squarefree quadratic initial ideal if $H$ is either a tree or a cycle.
Thus, by Proposition \ref{tfiber}, $I_{\QQ_G}$ has a squarefree quadratic initial ideal
if $G$ is a ring graph.
\end{proof}

A graph $G$ is called {\em perfect} if 
the chromatic number of every induced subgraph of $G$ equals the size of the largest clique of that subgraph.
See, e.g., \cite{Diestel}.
The following is known (\cite{GPT, OHcompressed, Sul}).

\begin{Proposition}
\label{perfectcompressed}
Let $G$ be a simple graph.
Then, $\QQ_G$ is compressed if and only if $G$ is perfect.
In particular, if $G$ is perfect, then $\QQ_G$ is normal.
\end{Proposition}

For a graph $G$ on the vertex set $[n]$, let $\overline{G}$ denote the complement of a graph $G$.
An induced cycle of $G$ of length $>3$ is called {\em hole} of $G$.
An induced cycle of $\overline{G}$ of length $>3$ is called {\em antihole} of $G$.
%For a graph $G$ with $\alpha(G) = 2$,
The first statement of Proposition \ref{perfecttwo} is called the {\em strong perfect graph theorem}.

\begin{Proposition}
\label{perfecttwo}
Let $G$ be a simple graph.
Then, $G$ is a perfect graph
if and only if $G$ has no odd holes and no odd antiholes. 
In particular, $G$ is a perfect graph with $\alpha(G)=2$
if and only if $\overline{G}$ is bipartite and not empty.
\end{Proposition}

For a graph $G$, let $\overline{G}^\star $ be the nonsimple graph on the vertex set $[n+1]$ 
whose edge (and loop) set is 
$
E(\overline{G}) \cup 
\{
\{i,n+1\}\ : \ i \in [n+1]\}.
$
The following lemma will play an important role
when we study the stable set polytope $\QQ_G$
of $G$.

\begin{Lemma}
\label{keylemma}
Let $G$ be a simple graph with $\alpha(G) = 2$.
Then,  we have $K[\QQ_G] \simeq K[\PP_{ \overline{G}^\star   }]$.
Moreover, if $\overline{G}$ is bipartite, then there exists a bipartite graph $H$
such that $K[\QQ_G] \simeq K[\PP_H]$.
\end{Lemma}

\begin{proof}
Let $\varphi : K[\QQ_G] \rightarrow K[x_1,\ldots,x_{n+1}]$
be the injective ring homomorphism defined by
$
\varphi (x_1^{a_1} \cdots x_n^{a_n}t ) =
x_1^{a_1} \cdots x_n^{a_n} x_{n+1}^{2-(a_1+\cdots + a_n)}.
$
Then,
$
\varphi (t ) = x_{n+1}^2
$, 
$
\varphi (x_i t ) =
x_i x_{n+1}
$
for $i = 1,2,\ldots,n$, and
$
\varphi (x_k x_\ell t ) =
x_k x_\ell
$ for each stable set $\{k,\ell\}$ of $G$.
Note that $\{k,\ell\}$ is a stable set of $G$
if and only if $\{k,\ell\}$ is an edge of $\overline{G}$.
Hence, the image of $\varphi$ is $K[\PP_{ \overline{G}^\star   }]$.

Suppose that $\overline{G}$ is bipartite.
Then, $\overline{G}$ has no odd cycles.
Hence, any odd cycle of $\overline{G}^\star $ has the vertex $n+1$.
(Note that $\{n+1,n+1\}$ is an odd cycle of length 1.)
Thus, in particular, any two odd cycles of $\overline{G}^\star $
has a common vertex.
By the argument in \cite[Proof of Proposition 5.5]{OHcentral}, 
there exists a bipartite graph $H$
such that $K[\QQ_G] \simeq K[\PP_H]$.
\end{proof}

The first application of Lemma \ref{keylemma} is as follows:

\begin{Proposition}
\label{unimodular}
Let $G$ be a simple graph.
Then, the following conditions are equivalent:
\begin{enumerate}
\item[(i)]
$\QQ_G$ is unimodular;
%(i.e., all triangulations of $\QQ_G$ are unimodular);
\item[(ii)]
$\overline{G}$ is bipartite.
\end{enumerate}
Moreover, if $\alpha(G) = 2$,
then the conditions
\begin{enumerate}
\item[(iii)]
$\QQ_G$ is compressed;
\item[(iv)]
$G$ is perfect
\end{enumerate}
are also equivalent to conditions {\rm (i)} and {\rm (ii)}.
\end{Proposition}

\begin{proof}
We may assume that $G$ is not complete (i.e., $\overline{G}$ is not empty
and $\alpha(G) \ne 1$).
Let $A$ be the matrix whose columns are
vertices of $\QQ_G$
and 
let 
$$
\widetilde{A}
=
\left(
\begin{array}{ccc}
 & A & \\
1 & \cdots & 1 
\end{array}
\right)
\mbox{ and } \ 
B=
\left(
\begin{array}{cccc}
 {\bf 0} & {\bf e}_1 & \cdots & {\bf e}_n\\
1 & 1 & \cdots & 1 
\end{array}
\right)
.$$
Then,
% ${\bf 0}, {\bf e}_1, \ldots, {\bf e}_n$ are the vertices of $\QQ_G$,
$B$ is a submatrix of $\widetilde{A}$.
Since $| \det (B) | =1$,
it is known \cite[p.70]{Stu} that
$\QQ_G$ is unimodular if and only if 
the absolute value of 
any nonzero $(n+1)$-minor of the matrix $\widetilde{A}$
is 1.

Suppose that $\overline{G}$ is not bipartite.
Then, $\overline{G}$ has an odd cycle $C=(i_1,\ldots,i_{2\ell+1})$.
The the absolute value of the $(n+1)$-minor of $\widetilde{A}$
that corresponds to
$$
\{ {\bf e}_{i_k} +  {\bf e}_{i_{k+1}} : 1 \leq k  \leq 2\ell \}
\cup
\{  {\bf e}_{i_1} +  {\bf e}_{i_{2\ell+1}}, {\bf 0} \}
\cup
\{
{\bf e}_j  :  j \notin 
\{  i_1,\ldots,i_{2\ell+1} \}
\}
$$
equals to 2.
Hence, $\QQ_G$ is not unimodular.
Thus, we have (i) $\Rightarrow$ (ii).

Suppose that $\overline{G}$ is bipartite.
By Lemma \ref{keylemma}, there exists a bipartite graph $H$ such that
$K[\QQ_G] \simeq  K[\PP_H]$.
It is well-known that the edge polytope of a bipartite graph is unimodular.
Thus, $\PP_H$ is unimodular and hence we have (ii) $\Rightarrow$ (i).

Suppose that $\alpha(G)=2$.
By Proposition \ref{perfectcompressed}, conditions (iii) and (iv) are equivalent.
In addition, by Proposition \ref{perfecttwo}, conditions (ii) and (iv) are equivalent.
\end{proof}

We close this section with the fundamental fact on stable set and edge polytopes.

\begin{Proposition}
\label{cpure}
Let $G'$ be an induced subgraph of a graph $G$.
Then,
\begin{itemize}
\item[(i)]
The edge polytope $\PP_{G'}$ is a face of $\PP_G$;
\item[(ii)]
If $G$ is a simple graph, then 
$\QQ_{G'}$ is a face of $\QQ_G$.
\end{itemize}
\end{Proposition}

By Proposition~\ref{cpure}, a lot of the properties of $K[\QQ_G]$
(resp.~$K[\PP_G]$)
are inherited to $K[\QQ_{G'}]$ (resp.~$K[\PP_{G'}]$);
for example, the normality of the toric ring, the existence of 
a squarefree initial ideal, the existence of a quadratic Gr\"obner basis,
the existence of the set of quadratic binomial generators of the toric ideal.
See \cite{cps}.

\section{Normality of stable set polytopes}

In this section, we study the normality of stable set polytopes.
The normality of an edge polytope studied in \cite{OHnormal, SVV} will play an important role.
If $C_1$ and $C_2$ are cycles in a graph $G$,
% having no common vertex,
then an edge $\{i, j\}$ of $G$ is called a {\em bridge} of $C_1$ and $C_2$
if $i$ is a vertex of $C_1 \setminus C_2$ and if $j$ is a vertex of $C_2 \setminus C_1$.
We say that a graph $G$ satisfies the {\em odd cycle condition} if,
for arbitrary two induced odd cycles $C_1$ and $C_2$ in $G$,
either $C_1$ and $C_2$ have a common
vertex or there exists a bridge of $C_1$ and $C_2$.
For the sake of simpleness, assume that a graph $H$ has at most one loop.
Then, it is known \cite{OHnormal, SVV} that $\PP_H$ is normal if and only if 
each connected component of $H$ satisfies the odd cycle condition.
By \cite[Corollary 2.3]{OHnormal} and Lemma \ref{keylemma},
we have the following.
(Note that $\overline{G}$ below is not necessarily connected.)

\begin{Theorem}
\label{a2normal}
Let $G$ be a simple graph with $\alpha(G) = 2$.
Then, the following conditions are equivalent.
\begin{enumerate}
\item[{\rm (i)}]
$\QQ_G$ is normal{\rm ;}
\item[{\rm (ii)}]
$\QQ_G$ has a unimodular covering{\rm ;}
\item[{\rm (iii)}]
$\overline{G}$ satisfies the odd cycle condition, i.e.,
if two odd holes $C_1$ and $C_2$ in $\overline{G}$ have no common
vertices, then there exists a bridge of $C_1$ and $C_2$ in $\overline{G}$.
\end{enumerate}
In particular, if $\QQ_G$ is normal, then $\PP_{\overline{G}}$ is normal.
\end{Theorem}

\begin{proof}
Let $G$ be a simple graph on the vertex set $[n]$
with $\alpha(G) = 2$.
By Lemma \ref{keylemma}, we have $K[\QQ_G]
\simeq K[\PP_{\overline{G}^\star}]$.
Hence, by \cite[Corollary 2.3]{OHnormal},
conditions (i) and (ii) are equivalent, and they hold if and only if
$\overline{G}^\star$ satisfies the odd cycle condition.
(Note that $\overline{G}^\star$ is connected.)
Since the vertex $n+1$ is incident with any vertices of 
 $\overline{G}^\star$,
it is easy to see that  $\overline{G}^\star$
satisfies the odd cycle condition if and only if 
$\overline{G}$ satisfies the odd cycle condition.
\end{proof}

It is known \cite{hanrei} that there exists a graph $G$ such that 
$\PP_G$ is normal and that $I_{\PP_G}$ has no squarefree initial ideals.
In \cite{OhsugiInfinity}, infinite families of  such edge polytopes are given.
We can construct the stable set polytopes with the same properties.
Let $G_1 (p_1,\ldots, p_5)$ be a graph defined in \cite[Theorem 3.10]{OhsugiInfinity}.

\begin{Theorem}
\label{aninfinite}
Let $G$ be a graph whose complement is $G_1 (p_1,\ldots, p_5)$ with $p_i \ge 2$ for $i=1,\ldots, 5$.
Then, $\QQ_G$ is normal and
$I_{\QQ_G}$ has no squarefree initial ideals.

\end{Theorem}

\begin{proof}
Since $\overline{G}$ has no triangles, we have $\alpha(G) = 2$.
Since $G_1 (p_1,\ldots, p_5)$ satisfies the odd cycle condition, 
$\QQ_G$ is normal by Theorem \ref{a2normal}.
On the other hand, $I_{\overline{G}}$ has no squarefree initial ideals.
Since $\overline{G}$ is an induced subgraph of $\overline{G}^\star$,
$I_{\QQ_G}$ has no squarefree initial ideals
by Lemma \ref{keylemma} and Proposition \ref{cpure}.
\end{proof}

It seems to be a challenging problem to characterize the normal steble set polytopes
with large stability number.
We give several necessary conditions.
The following is a consequence of Proposition \ref{cpure} and Theorem \ref{a2normal}.

\begin{Proposition}
Let $G$ be a simple graph.
Suppose that $\QQ_G$ is normal.
Then, any two odd holes of $\overline{G}$ without common vertex have a bridge in $\overline{G}$.
\end{Proposition}

\begin{proof}
Suppose that two odd holes $C_1$, $C_2$ of $\overline{G}$ without common vertices have no bridges in $\overline{G}$.
Let $H$ be an induced subgraph of $G$ whose vertex set is that of $C_1 \cup C_2$.
Then, $\alpha(H) =2$ and hence $\QQ_H$ is not normal by Theorem \ref{a2normal}.
Thus, $\QQ_G$ is not normal by Proposition \ref{cpure}.
\end{proof}

Similar conditions are required for antiholes of $\overline{G}$.

\begin{Proposition}
Let $G$ be a simple graph.
Suppose that $\QQ_G$ is normal.
Then, $G$ satisfies all of the following conditions:
\begin{enumerate}
\item[(i)]
Any two odd antiholes of $\overline{G}$ having no common vertices have a bridge in $\overline{G}$.
\item[(ii)]
Any two odd antiholes of $\overline{G}$ of length $\ge 7$ having exactly one common vertex have a bridge in $\overline{G}$.
\item[(iii)]
Any odd hole and odd antihole of $\overline{G}$ having no common vertices have a bridge in $\overline{G}$.
\end{enumerate}
\end{Proposition}

\begin{proof}
Let $G$ be a graph on the vertex set $[n]$.
Let 
$
{\mathcal A} = \{ (\rho(W),1)^t : W \in S(G) \} %= \{ \alpha_1,  \ldots, \alpha_m \}
$.
It is known \cite[Proposition 13.5]{Stu} that 
$\QQ_G$ is normal if and only if 
we have $\ZZ_{\ge 0} {\mathcal A}  = {\mathbb Q}_{\ge 0} {\mathcal A} \cap \ZZ^{n+1}$.

(i)
Let $C_1 = (i_1,\ldots, i_{2k +1})$ and $C_2 = (j_1,\ldots, j_{2\ell +1})$ be
 odd antiholes in $\overline{G}$ having no common vertices and no bridges in $\overline{G}$.
By Proposition \ref{cpure}, we may assume that $\overline{G} = C_1 \cup C_2$.
Then, 
\begin{eqnarray*}
\sum_{W \in S(\overline{C_1}) \mbox{ and } |W| = k} (\rho(W) ,1) &=& (2k+1) \eb_{n+1} + k \sum_{p=1}^{2k+1} \eb_{i_p},\\
\sum_{W \in S(\overline{C_2}) \mbox{ and } |W| = \ell} (\rho(W),1 )  &=& (2\ell+1) \eb_{n+1} + \ell \sum_{q=1}^{2\ell+1} \eb_{j_q}.
\end{eqnarray*}
Since $k, \ell \ge 2$, we have $k \ell - k - \ell \ge 0$.
Hence,
\begin{eqnarray*}
 & & 5 \eb_{n+1} + \sum_{p=1}^{2k+1} \eb_{i_p} +  \sum_{q=1}^{2\ell+1} \eb_{j_q}\\
&=&
\frac{1}{k} \left((2k+1) \eb_{n+1} + k \sum_{p=1}^{2k+1} \eb_{i_p}\right)
+
\frac{1}{\ell} \left((2\ell+1) \eb_{n+1} + \ell \sum_{q=1}^{2\ell+1} \eb_{j_q}\right)
+
\frac{k \ell - k - \ell}{k \ell} \eb_{n+1}
\end{eqnarray*}
belongs to ${\mathbb Q}_{\ge 0} {\mathcal A} \cap \ZZ^{n+1}$.
However, this vector is not in $\ZZ_{\ge 0} {\mathcal A}$
since $\max \{ |W| : W \in S(\overline{C_1})  \} = k$, $\max \{ |W| : W \in S(\overline{C_2})  \} = \ell$, and
$\lceil (2k+1) /k \rceil + \lceil (2\ell+1) /\ell \rceil = 6 >5$.

(ii)
Let $C_1 = (i_1,\ldots, i_{2k +1})$ and $C_2 = (j_1,\ldots, j_{2\ell +1})$ be
 odd antiholes in $\overline{G}$ of length $\ge 7$ having exactly one common vertex $i_1 = j_1$ and no bridges in $\overline{G}$.
By Proposition \ref{cpure}, we may assume that $\overline{G} = C_1 \cup C_2$.
Let
\begin{eqnarray*}
S_1 &=& \{ W \in S(\overline{C_1}) : |W| = k \mbox{ and either } i_1 \in W \mbox{ or } \{i_2,i_{2k+1}\} \subset W \},\\
S_2 &=& \{ W \in S(\overline{C_2}) : |W| = \ell \mbox{ and either } j_1 \in W \mbox{ or } \{j_2,j_{2\ell+1}\} \subset W \}.
\end{eqnarray*}
Then, 
\begin{eqnarray*}
\sum_{W \in S_1} (\rho(W) ,1) &=& (2k-1) \eb_{n+1} + k \eb_{i_1} + (k-1) \sum_{p=2}^{2k+1} \eb_{i_p},\\
\sum_{W \in S_2} (\rho(W),1 )  &=& (2\ell-1) \eb_{n+1} + \ell \eb_{i_1} + (\ell-1) \sum_{q=2}^{2\ell+1} \eb_{j_q}.
\end{eqnarray*}
Since $k, \ell \ge 3$, we have $0 < 1/(k-1) + 1/(\ell-1) \le 1$.
Hence,
\begin{eqnarray*}
\alpha & :=& 5 \eb_{n+1} + 3 \eb_{i_1} + \sum_{p=2}^{2k+1} \eb_{i_p} +  \sum_{q=2}^{2\ell+1} \eb_{j_q}\\
&\; =&
\frac{1}{k-1} \left( (2k-1) \eb_{n+1} + k \eb_{i_1} + (k-1) \sum_{p=2}^{2k+1} \eb_{i_p}\right)\\
 & & +
\frac{1}{\ell-1} \left((2\ell-1) \eb_{n+1} + \ell \eb_{i_1} + (\ell-1) \sum_{q=2}^{2\ell+1} \eb_{j_q}\right)
+
\left( 1- \frac{1}{k-1}-\frac{1}{\ell-1} \right) (\eb_{i_1} + \eb_{n+1})
\end{eqnarray*}
belongs to ${\mathbb Q}_{\ge 0} {\mathcal A} \cap \ZZ^{n+1}$.
Suppose that $\alpha = (\rho(W_1), 1) + \cdots +    (\rho(W_5), 1) $
for $W_i \in S(\overline{G})$.
Then, each $W_i$ belongs to either $S(\overline{C_1})$ or $S(\overline{C_2})$.
Since $\max \{ |W| : W \in S(\overline{C_1})  \} = k$, $\max \{ |W| : W \in S(\overline{C_2})  \} = \ell$, 
we have
\begin{eqnarray*}
|\{W_1,\ldots, W_5\} \cap S(\overline{C_1})| &\ge& 2,\\
|\{W_1,\ldots, W_5\} \cap S(\overline{C_2})| &\ge& 2.
\end{eqnarray*}
We may assume that $W_1, W_2 \in S(\overline{C_1})$ and $W_3,W_4,W_5 \in S(\overline{C_2})$.
It then follows that $\rho(W_1) + \rho(W_2) = \sum_{p=2}^{2k+1} \eb_{i_p}$, and hence
$\rho(W_3) + \rho(W_4) + \rho(W_5) =   3 \eb_{i_1} + \sum_{q=2}^{2\ell+1} \eb_{j_q} $.
Hence $i_1 \in W_3 \cap W_4 \cap W_5$.
Thus, $i_2, i_{2\ell+1} \notin W_3, W_4, W_5$, a contradiction.
Therefore, we have $\alpha \notin \ZZ_{\ge 0} {\mathcal A}$.

(iii)
Let $C_1 = (i_1,\ldots, i_{2k +1})$ be an odd hole and $C_2 = (j_1,\ldots, j_{2\ell +1})$ an odd antihole in $\overline{G}$ having no common vertices.
By Proposition \ref{cpure}, we may assume that $\overline{G} = C_1 \cup C_2$.
Then, 
\begin{eqnarray*}
\sum_{W \in S(\overline{C_1}) \mbox{ and } |W| = 2}  (\rho(W) ,1) &=& (2k+1) \eb_{n+1} + 2\sum_{p=1}^{2k+1} \eb_{i_p},\\
\sum_{W \in S(\overline{C_2}) \mbox{ and } |W| = \ell}   (\rho(W) ,1)   &=& (2\ell+1) \eb_{n+1} + \ell \sum_{q=1}^{2\ell+1} \eb_{j_q}.
\end{eqnarray*}
Hence, 
\begin{eqnarray*}
 & & (k+3) \eb_{n+1} + \sum_{p=1}^{2k+1} \eb_{i_p} +  \sum_{q=1}^{2\ell+1} \eb_{j_q}\\
&=&
\frac{1}{2} \left((2k+1) \eb_{n+1} + 2 \sum_{p=1}^{2k+1} \eb_{i_p}\right)
+
\frac{1}{\ell} \left((2\ell+1) \eb_{n+1} + \ell \sum_{q=1}^{2\ell+1} \eb_{j_q}\right)
+
\frac{\ell - 2}{2 \ell} \eb_{n+1}
\end{eqnarray*}
belongs to ${\mathbb Q}_{\ge 0} {\mathcal A} \cap \ZZ^{n+1}$.
However, this vector is not in $\ZZ_{\ge 0} {\mathcal A}$
since $\max \{ |W| : W \in S(\overline{C_1})  \} = 2$, $\max \{ |W| : W \in S(\overline{C_2})  \} = \ell$,
and
$\lceil (2k+1) /2 \rceil + \lceil (2\ell+1) /\ell \rceil =k+4 >k+3$.
\end{proof}

Unfortunately, the above conditions are not sufficient to be normal in general.
For example, if the length of the two odd antiholes of $\overline{G}$ without common vertices are long, then a lot of bridges 
in $\overline{G}$ seem to be needed.

\section{Generators and Gr\"obner bases of $I_{\QQ_G}$}

For a toric ideal $I$, let $\mu(I)$ be the maximum degree of binomials in a minimal set of binomial generators of $I$.
If $I=\{0\}$, then we set $\mu(I)=0$.
In this section, we study $\mu(I_{\QQ_G})$ by using results on the toric ideals of edge polytopes.

Let $G$ be a graph on the vertex set $[n]$ allowing loops and having no multiple edges.
Let $E(G)=\{e_1, \ldots, e_m\}$ be a set of all edges and loops of $G$.
The toric ideal $I_{\PP_G}$ is the kernel of a homomorphism
$\pi : K[y_1,\ldots, y_m] \rightarrow K[x_1,\ldots, x_n]$
defined by $\pi(y_i) = x_k x_\ell$ where $e_i = \{k,\ell \}$.
A walk of length $q$ of $G$ connecting $v_1\in [n]$ and $v_{q+1} \in[n]$ is 
a finite sequence of the form
\begin{equation}
\label{walk}
\Gamma =
 (
\{v_1,v_2\},
\{v_2,v_3\},
\ldots,
\{v_{q},v_{q+1}\}
)
\end{equation}
with each $\{v_k,v_{k+1}\} \in E(G)$.
An {\em even} (resp.~{\em odd}) {\em walk} is a walk of even (resp.~odd) length.
A walk $\Gamma$ of the form (\ref{walk}) is called {\em closed} if $v_{q+1} =v_1$. 
Given an even closed walk
$
\Gamma = (e_{i_1},e_{i_2},\ldots,e_{i_{2q}})
$
of $G$, we write $f_\Gamma$ for the binomial
$$
f_\Gamma = 
\prod_{k=1}^q y_{i_{2k-1}} -
\prod_{k=1}^q y_{i_{2k}}
\in I_{\PP_G}.
$$
We regard a loop as an odd cycle of length 1.
It is known (\cite{OHquad, Stu, Vil}) that 

\begin{Proposition}
\label{ecw}
Let $G$ be a graph having at most one loop.
Then, $I_{\PP_G}$ is generated by all the binomials
$
f_\Gamma
$,
where 
$\Gamma$ is an even closed walk of $G$.
In particular, $I_{\PP_G}=(0)$ if and only if 
each connected component of $G$ has at most one cycle and the cycle is odd.
\end{Proposition}

The following theorem implies that the set of binomial generators of $I_{\QQ_G}$ is also
characterized by the graph theoretical terminology if $\alpha(G)=2$.

\begin{Theorem}
\label{generators}
Let $G$ be a simple graph with $\alpha(G) = 2$.
Then, $I_{\QQ_G} = I_{\PP_{\overline{G}}} + J$ where $J$ is an ideal generated by 
quadratic binomials $f_\Gamma$ where $\Gamma$
is an even closed walk of $\overline{G}^\star$ that satisfies one of the following:
\begin{enumerate}
\item[(i)]
$\Gamma = (\{i,j\},\{j,k\},\{k,n+1\},\{n+1,i\})$ is a cycle where $\{i,j\}$, $\{j,k\} \in E(\overline{G})$;
%$\Gamma$ is a cycle of length $4$ containing the vertex $n+1$;
\item[(ii)]
$\Gamma = (\{i,j\},\{j,n+1\},\{n+1,n+1\},\{n+1,i\})$ where $\{i,j\} \in E(\overline{G})$.
\end{enumerate}
In particular,
$\mu(I_{\QQ_G}) = 
\max\{ 
\mu(I_{\PP_{\overline{G}}}) , 2
\}$.
%In particular, 
%$\mu(I_{\QQ_G})=2$ if and only if 
%either $\mu(I_{\PP_{\overline{G}}}) =2$
%or $I_{\PP_{\overline{G}}} =\{0\}$.
\end{Theorem}

\begin{proof}
Let $\Gamma$ be an even closed walk of $\overline{G}^\star$.
It is enough to show that $f_\Gamma \in I_{\PP_{\overline{G}^\star}}$ 
%is generated by the set of binomial generators of 
belongs to $I_{\PP_{\overline{G}}}+J$.
Suppose that $f_\Gamma$ 
%is not generated by the set of binomial generators of 
does not belong to $I_{\PP_{\overline{G}}}+J$.
Then, the vertex $n+1$ belongs to $\Gamma$.
We may assume that the degree of $f_\Gamma$ is 
minimum among such binomials.
Then, $f_\Gamma$ is irreducible.
Let 
$$\Gamma = (\{n+1,p\},\{p,q\},\{q,r\},\Gamma').$$
Since $f_\Gamma$ is irreducible, it follows that
$p , q,  r$ are distinct each other and that $q \ne n+1$.
Then, $f_\Gamma$ is generated by $f_{\Gamma_1}$
and $f_{\Gamma_2}$ where $\Gamma_1 = (\{n+1,p\},\{p,q\},\{q,r\},\{r, n+1\})$ and $\Gamma_2 = (\{n+1,r\},\Gamma')$.
Since $\deg f_{\Gamma_2} < \deg f_{\Gamma}  $, the binomial
$f_{\Gamma_2} $ belongs to $I_{\PP_{\overline{G}}} +J$
by the assumption on $\deg f_\Gamma$.
Moreover, 
$\Gamma_1$ satisfies one of the conditions (i) or (ii)
and hence $f_{\Gamma_1} \in J$.
Thus, we have $f_\Gamma \in \langle f_{\Gamma_1},f_{\Gamma_2} \rangle \subset I_{\PP_{\overline{G}}} +J$,
a contradiction.
\end{proof}

\begin{Remark}
A graph theoretical characterization of a {\em simple} graph $G$
such that $\mu(I_{\PP_G}) \le 2$ is given in 
\cite{OHquad}.
By Theorem \ref{generators},
the characterization \cite[Theorem 1.2]{OHquad} gives a graph theoretical characterization of a simple graph $G$
such that $\alpha(G) =2$ and $\mu(I_{\QQ_G}) \le 2$.
\end{Remark}

It is known \cite[Theorem 7.4]{EnNo} that, 
if the complement of a graph $G$ is an even cycle of length $2k$, then
we have $\mu(I_{\QQ_G}) = k$.
By Theorem \ref{generators}, we can generalize this result 
for a graph whose complement is an arbitrary bipartite graph.

\begin{Corollary}
\label{KB}
Let $G$ be a simple graph such that $\overline{G}$ is bipartite.
Then, we have 
$$
\mu(I_{\QQ_G}) = 
\left\{
\begin{array}{cl}
0 &\mbox{ if } \overline{G} \mbox{ is empty (i.e., } G \mbox{ is complete),}\\
k &\mbox{ if } \overline{G} \mbox{ has a cycle},\\
2 &\mbox{ otherwise,}
\end{array}
\right.
$$
where $2k$ is the maximum length of induced cycles of $\overline{G}$.
Moreover, the following conditions are equivalent:
\begin{enumerate}
\item[(i)]
$\mu(I_{\QQ_G}) \le 2$, i.e.,
$I_{\QQ_G}$ is generated by quadratic binomials;
\item[(ii)]
$K[\QQ_G]$ is Koszul;
\item[(iii)]
$I_{\QQ_G}$ has a quadratic Gr\"obner basis;
\item[(iv)]
The length of any induced cycle of $\overline{G}$ is $4$.
\end{enumerate}
\end{Corollary}

\begin{proof}
Let $G$ be a simple graph such that $\overline{G}$ is bipartite.
Then, $\alpha(G)=2$.
Since $\overline{G}$ is bipartite, it is known (see, e.g., \cite{bipartitecycle}) that,
$I_{\PP_{\overline{G}}}$ is generated by $f_\Gamma$ 
where $\Gamma$ is an induced even cycle of $\overline{G}$.
Note that $\deg f_\Gamma = k$ if the length of $\Gamma$ is $2k$.
Hence, by Theorem \ref{generators}, we have a formula for $\mu(I_{\QQ_G})$.

By this formula, it follows that (i) and (iv) are equivalent.
Moreover, (iii) $\Rightarrow$ (ii) $\Rightarrow$ (i) holds in general.
By Lemma \ref{keylemma}, there exists a bipartite graph $H$
such that $K[\QQ_G] \simeq K[\PP_H]$.
By the theorem in \cite{OHbipartite}, $I_{\PP_H}$ has a quadratic Gr\"obner basis
if and only if $\mu(I_{\PP_H}) \le 2$.
Thus, we have (i) $\Rightarrow$  (iii).
\end{proof}

If $\overline{G}$ is not bipartite, 
then condition (i) and (iii) in Corollary \ref{KB} are not equivalent.
In order to construct an infinite family of counterexamples,
the following proposition is important.
(Proof is essentially given in \cite[Proof of Proposition 1.6]{OHquad}.)

\begin{Proposition}
\label{zeroone}
Let $P$ be a $(0,1)$-polytope.
If $I_P$ has a quadratic Gr\"obner basis, then the initial ideal is generated by 
squarefree monomials and hence $P$ is normal.
\end{Proposition}

\begin{Theorem}
\label{twooddholes}
Let $G$ be a simple graph such that $\overline{G}$ consists of two odd holes
without common vertices.
Then, $\alpha(G)=2$ and 
\begin{enumerate}
\item[(a)]
$I_{\QQ_G}$ is generated by quadratic binomials;
\item[(b)]
$I_{\QQ_G}$ has no quadratic Gr\"obner bases;
\item[(c)]
$\QQ_G$ is not normal. 
\end{enumerate}

\end{Theorem}

\begin{proof}
Since $\overline{G}$ has no triangles, we have $\alpha(G)=2$.
By Proposition \ref{ecw} and Theorem \ref{generators}, 
it follows that
$I_{\overline{G}} = \{0\}$ and $I_{\QQ_G}$ is generated by
quadratic binomials.
On the other hand, by Theorem \ref{a2normal}, $\QQ_G$ is not normal. 
Thus, by Proposition \ref{zeroone}, $I_{\QQ_G}$ has no quadratic Gr\"obner bases.
\end{proof}

The graphs in Theorem \ref{twooddholes} are not strongly Koszul 
by the result in \cite{KM}.
However, we do not know whether they are Koszul or not in general.

\begin{Remark}
It is known \cite[Theorem 1.2]{OHquad} that, 
if $G$ is a simple connected graph and $I_{\PP_G}$ is generated by quadratic binomials, 
then $G$ satisfies the odd cycle condition and hence $\PP_G$ is normal.
\end{Remark}

It seems to be a challenging problem to characterize the graphs $G$ such that $\alpha(G) >2$ and $\mu(I_{\QQ_G}) \le 2$.
The following is a consequence of Proposition \ref{cpure},
Theorem \ref{generators} and \cite[Theorem 1.2]{OHquad}.

\begin{Proposition}
Let $G$ be a simple graph.
If $I_{\QQ_G}$ is generated by quadratic binomials,
then $\overline{G}$ satisfies the following conditions:
\begin{enumerate}
\item[(i)]
Any even cycle of $\overline{G}$ of length $\geq 6$
has a chord;
\item[(ii)]
Any two odd holes of $\overline{G}$ having exactly one common vertex have a bridge;
\item[(iii)]
Any two odd holes of $\overline{G}$ having no common vertex have at least two bridges.

\end{enumerate}
\end{Proposition}

\begin{proof}
Suppose that $\overline{G}$ does not satisfy one of the conditions above.
If $\overline{G}$ does not satisfy condition (i), then
let $H$ be an induced subgraph of $G$ whose vertex set is that of the even cycle.
If $\overline{G}$ does not satisfy condition either (ii) or (iii), then
let $H$ be an induced subgraph of $G$ whose vertex set is that of two odd holes.
Then, $\alpha(H) =2$ and hence $I_{\QQ_H}$ is not generated by quadratic binomials
by Theorem \ref{generators} and \cite[Theorem 1.2]{OHquad}.
Thus, it follows from Proposition \ref{cpure} that $I_{\QQ_G}$ is not generated by quadratic binomials.
\end{proof}

\end{document}